\documentclass[a4paper,12pt]{amsart}

\makeatletter

{\catcode`\|=\active 
  \gdef\Braket#1{\left<\mathcode`\|"8000\let|\BraVert {#1}\right>}}
\def\BraVert{\egroup\,\mid@vertical\,\bgroup}
{\catcode`\|=\active
  \gdef\set#1{\mathinner{\lbrace\,{\mathcode`\|"8000\let|\midvert #1}\,\rbrace}}
  \gdef\Set#1{\left\{\:{\mathcode`\|"8000\let|\SetVert #1}\:\right\}}}
\def\midvert{\egroup\mid\bgroup}
\def\SetVert{\egroup\;\mid@vertical\;\bgroup}
\begingroup
 \edef\@tempa{\meaning\middle}
 \edef\@tempb{\string\middle}
\expandafter \endgroup \ifx\@tempa\@tempb
 \def\mid@vertical{\middle|}
\else
 \let\mid@vertical\vrule
\fi
\makeatother

\usepackage{url}

\newcommand{\defit}[1]{{\em #1}}

\newcommand{\der}{\partial}

\newcommand{\ideal}[1]{S \cdot {#1}}

\newcommand{\kk}{\mathbb{K}}
\newcommand{\ZZ}{\mathbb{Z}}
\newcommand{\NN}{\mathbb{N}}

\newcommand{\AAA}{{\mathcal A}}
\newcommand{\BBB}{{\mathcal B}}

\newcommand{\LAT}{{\Lambda}}
\newcommand{\Lat}{{\Lambda'}}

\newcommand{\numof}[1]{\left| #1 \right|}
\newcommand{\distance}[2]{d({#1},{#2})}
\newcommand{\disk}[2]{B({#1},{#2})}

\newcommand{\coveredby}{\mathrel{\dot\subset}}

\newcommand{\cone}[2]{#1_{#2}}
\newcommand{\innercone}[1]{\cone{#1}{\emptyset}}

\newcommand{\cc}[1]{{\operatorname{cc} (#1)}}
\newcommand{\fcc}[1]{{\operatorname{cc}_{0} (#1)}}
\newcommand{\icc}[1]{{\operatorname{cc}_\infty (#1)}}

\newcommand{\mpts}[1]{{P(#1)}}

\newcommand{\downalpha}[1]{{\alpha^{(#1)}}}

\newcommand{\sectional}[3]{{#1}_{#2,#3}}

\newcommand{\Sym}{\operatorname{Sym}}
\newcommand{\Der}[2]{\operatorname{Der}_{#1}(#2)}

\theoremstyle{plain}
\newtheorem{thm}{Theorem}[section]
\newtheorem{lemma}[thm]{Lemma}
\newtheorem{cor}[thm]{Corollary}
\newtheorem{theorem}[thm]{Theorem}
\newtheorem{proposition}[thm]{Proposition}

\theoremstyle{remark}
\newtheorem{remark}[thm]{Remark}

\newtheorem{example}[thm]{Example}
\newtheorem*{acknowledgments}{Acknowledgments}

\theoremstyle{definition}
\newtheorem{definition}[thm]{Definition}

\begin{document}

\title[Multiplicity Lattices]{Exponents of $2$-multiarrangements and multiplicity lattices}

\author[Abe, T.]{ABE, Takuro}
\address[Abe, T.]{Department of Mechanical Engineering and Science, Kyoto University\\
Sakyo-ku, Kyoto, 606-8501, Japan}

\author[Numata, Y.]{NUMATA, Yasuhide}
\address[Numata, Y.]{ Department of Mathematical Informatics\\ The University of Tokyo\\
Hongo 7-3-1, Bunkyo-ku, Tokyo, 113-8656, Japan }
\address[Numata, Y.]{ JST CREST }

\keywords{hyperplane arrangements; multiarrangements; exponents of derivation modules, multiplicity lattices}
\subjclass[2000]{32S22}

\begin{abstract}
We introduce a concept of multiplicity lattices of $2$-multiarrangements, 
determine the combinatorics and geometry of that lattice, and give 
a criterion and method to construct a basis for derivation modules 
effectively.
\end{abstract}

\maketitle
\section{Introduction}

Let $\kk$ be a field and  $V$  a two-dimensional vector space 
over $\kk$.  
Fix a basis $\Set{x,y}$ for $V^{\ast}$ and define 
$S:=\Sym(V^{\ast}) \simeq \kk[x,y]$.
A \defit{hyperplane arrangement} $\AAA$ 
is a finite collection of affine hyperplanes in $V$. 
In this article, 
we assume that any $H \in \AAA$ contains the origin. In other words, 
all hyperplane arrangements are \defit{central}. 
For each $H \in \AAA$, let us fix 
a linear form $\alpha_H \in V^{\ast}$ such that $\ker(\alpha_H)=H$. 
For a hyperplane arrangement $\AAA$, 
a map $\mu:\AAA \to \NN = \ZZ_{\geq 0}$
 is called a \defit{multiplicity} and 
a pair $(\AAA,\mu)$ a \defit{multiarrangement}. 
When we want to make it clear that all multiarrangements are 
considered in $V \simeq \kk^2$, we use the term 
\defit{$2$-multiarrangement}.
(Ordinarily, a $2$-multiarrangement is defined as
a pair $(\AAA,m)$ of a central hyperplane arrangement $\AAA$ and multiplicity function $m:\AAA\to\ZZ_{>0}$.
From a $2$-multiarrangement $(\AAA,\mu)$ in our definition,
we can obtain a $2$-multiarrangement $(\AAA',m)$ in the original definition
by assigning
$\AAA' = \mu^{-1}(\ZZ_{>0})$ and 
$m=\mu|_{\AAA'}$. We identify ours with the original one in this manner.)
To each multiarrangement $(\AAA,\mu)$, 
we can associate the $S$-module $D(\AAA,\mu)$, 
called the \defit{derivation module} by the following manner:
\begin{gather*}
D(\AAA,\mu):=\Set{ \delta \in \Der{\kk}{S} |
\delta(\alpha_H) \in \ideal{\alpha_H^{\mu(H)}} \ (\forall H \in \AAA)},
\end{gather*}
where $\Der{\kk}{S} := S\cdot\der_x \oplus S\cdot\der_y$ is the
module of derivations. 
It is known that 
$D(\AAA,\mu)$ is a free graded $S$-module 
because we only consider $2$-multiarrangements 
(see \cite{Sa2}, \cite{OT} and \cite{Z}). 
If we choose a homogeneous basis $\Set{\theta,\theta'}$ for
$D(\AAA,\mu)$, 
then the \defit{exponents} of $(\AAA,\mu)$, 
denoted by $\exp(\AAA,\mu)$, 
is a multiset defined by
\begin{gather*}
\exp(\AAA):=(\deg (\theta),\deg (\theta')),
\end{gather*}
where the degree is a polynomial degree.

Multiarrangements were originally introduced by Ziegler in \cite{Z} 
and there are a lot of studies related to a multiarrangement 
and its derivation module. 
Especially, Yoshinaga characterized the freeness of hyperplane
arrangements by using the freeness of multiarrangements 
(\cite{Y1} and \cite{Y2}). 
In particular, according to the results in \cite{Y2}, 
we can obtain the necessary and sufficient condition 
for a hyperplane arrangement 
in three-dimensional vector space to be free 
in terms of the combinatorics of hyperplane arrangements, 
and the explicit description of exponents of $2$-multiarrangements. 
This is closely related to the Terao conjecture, 
which asserts that the freeness of hyperplane arrangements 
depends only on the combinatorics. 
However, instead of the simple description of the exponents 
of hyperplane arrangements, 
it is shown by Wakefield and Yuzvinsky in \cite{WY} 
that the general description of the exponents of $2$-multiarrangements 
are very difficult. 
In fact, 
there are only few results related 
to them (\cite{A0}, \cite{A} and \cite{W}). 
Recently, 
some theory to study the freeness of multiarrangements are developed 
by the first author, 
Terao and Wakefield in \cite{ATW1} and \cite{ATW2}, and 
some results on the free multiplicities are appearing (\cite{A2}). 
In these papers, 
the importance of the exponents of $2$-multiarrangements 
is emphasized too. 
Hence 
it is very important
to establish some general theory for the exponents of $2$-multiarrangements. 

The aim of this article is to give some answers to this problem. 
Our idea is to introduce the concept of the 
\defit{multiplicity lattice} of a fixed hyperplane arrangement.
The aim of the study of this lattice is similar, but the method is contrary to 
the study in \cite{WY}, for Wakefield and Yuzvinsky fixed one multiplicity and 
consider all hyperplane arrangements with it, but we fix one hyperplane arrangement and 
consider all multiplicities on it.
Let us fix a central hyperplane arrangement $\AAA=\Set{H_1,\ldots,H_n}$
and the lattice $\LAT=\NN^{\numof{\AAA}}$.
We identify $\mu\in\LAT$ with the map $\AAA\to\NN$ such that
$\mu(H_i)=\mu_i$ for $H_i\in \AAA$. 
Define a map $\Delta:\LAT \to \ZZ_{\geq 0}$ by 
\begin{gather*}
\Delta(\mu):=\deg(\theta'_\mu)-\deg (\theta_\mu),
\end{gather*}
where $\Set{\theta_\mu,\theta'_\mu }$ is a basis for 
$D(\AAA,\mu)$ such that $\deg (\theta_\mu)\leq \deg (\theta'_\mu)$. 
If we put $\Lat:=\LAT \setminus \Delta^{-1}(\Set{0})$, then 
$\theta_{\mu}$ is unique up to a scalar for each $\mu \in \Lat$, though 
$\theta'_\mu$ is not. 
Hence $\theta_\mu$ for $\mu \in \Lat$ is expected to have 
some good properties. Our main results are the investigations of 
these properties through considering 
the shape, topology and combinatorics of $\Lat$. 
For details, see Section 3, or Theorem \ref{theorem:misc}, 
Theorem \ref{theorem:shape} and Theorem \ref{theorem:independency}.
These results, combined with Saito's criterion 
(Theorem \ref{Saito}), allow us to 
construct a basis for $2$-multiarrangements effectively, 
see Theorem \ref{theorem:basis} for details.

Now the organization of this article is as follows. 
In Section \ref{section:def},
we introduce some notation and examples related to our new definitions.
In Section \ref{section:main},
we state the main results. 
In Section \ref{section:proof}, we 
recall elementary results about hyperplane arrangement theory
and prove the main results. 
In Section \ref{section:application},
 we show some 
applications of main results, especially determine some exponents of
multiarrangements of the Coxeter type. 

\begin{acknowledgments}
The authors are grateful
to the referee for pointing out several mistakes in the first draft and 
giving a lot of useful comments.
\end{acknowledgments}

\section{Definition and Notation}\label{section:def}
In this section, 
we introduce some basic terms and notation.
Let $\kk$ be a field, $V$ a two-dimensional vector space over $\kk$,
and $S$ a symmetric algebra of $V^{\ast}$. By choosing 
a basis $\Set{x,y}$ for $V^{\ast}$, $S$ can be identified with a polynomial ring $\kk[x,y]$.
The algebra $S$ can be graded by polynomial degree as $S=\bigoplus_{i\in\NN} S_i$, 
where $S_i$ is a vector space whose basis
is $\Set{x^j y^{i-j}| j=0,\ldots,i}$.

Let us fix a central hyperplane arrangement $\AAA$ in $V$, 
i.e., 
a finite collection $\Set{H_1,\ldots,H_n}$ of linear hyperplanes in $V$. 
For $H \in \AAA$, fix $\alpha_H \in S_1$ such that $\ker(\alpha_H)=H$. 
The following new definition plays the key role in this article.
\begin{definition}
We define \defit{the  multiplicity lattice}
$\LAT$ of $\AAA$
by 
\begin{gather*}
\LAT:=\NN^{\numof{\AAA}} = \NN^{n}.
\end{gather*}
\end{definition}
Let us identify $\mu=(\mu_1,\ldots,\mu_n) \in \LAT$ with the multiplicity 
$\mu:\AAA\to \NN$ 
defined by $\mu(H_i):=\mu_{H_i}=\mu_i$. Hence 
a pair $(\AAA,\mu)$ 
can be considered as a multiarrangement. 
The set $\LAT$ has 
the partial order $\subset$ defined by
\begin{gather*}
\mu \subset \nu \iff \text{$\mu_H \leq \nu_H$ for all $H$}.
\end{gather*}
For $\mu,\nu \in \LAT$, 
the binary operations $\wedge$ and $\vee$ are
defined by 
\begin{align*}
\mu \wedge \nu &:= \inf\Set{\mu, \nu},\\ 
\mu \vee \nu   &:= \sup\Set{\mu,\nu}, 
\end{align*}
i.e.,
$(\mu \wedge \nu)_H = \min\Set{\mu_H, \nu_H}$ and 
$(\mu \vee \nu)_H = \max\Set{\mu_H,\nu_H}$.
For $\mu \in \LAT$, 
we define the \defit{size}
$\numof{\mu}$ of $\mu$
by $\numof{\mu}:=\sum_{H\in \AAA} \mu_H$.
The element $0$, which is defined by $0_H = 0$ for all $H \in \AAA$,
is the minimum element.
The \defit{covering relation} $\mu \coveredby \nu$ is defined 
by $\mu \subset \nu$ and 
$\numof{\mu} + 1 = \numof{\nu}$.
The graph whose set of edges is $\Set{(\mu,\nu)\in\LAT^2|\mu\coveredby\nu}$ 
and whose set of vertices is $\LAT$ is called the Hasse graph of $\LAT$.
We identify $\LAT$ and its subset with (the set of vertices of)
the Hasse graph and its induced subgraph, respectively.
For $\mu,\nu\in \LAT$,
we define the \defit{distance} $\distance{\mu}{\nu}$ 
by 
$\distance{\mu}{\nu}:=\sum_{H\in \AAA} |\mu_H-\nu_H|$.
For $C, C' \subset \LAT$, 
we define $\distance{C}{C'}$ by
$\distance{C}{C'}:=\min\Set{\distance{\mu}{\mu'}|\mu\in C, \mu'\in C'}$.
For $\mu\in \LAT$ and $r\in\NN$,
we define the \defit{ball} $\disk{\mu}{r}$ with the radius $r$ and
center $\mu$  
by $\disk{\mu}{r}:=\Set{\nu\in \LAT| \distance{\mu}{\nu}<r}$.
\begin{definition}
We define a map $\Delta:\LAT \to \NN$ by
\begin{gather*}
\Delta(\mu) := |d_1 - d_2|,
\end{gather*}
where $(d_1,d_2)$ are the exponents of the 
free multiarrangement $(\AAA,\mu)$.
\end{definition}

\begin{definition}
Let $\Lat$ denote the support $\Delta^{-1}(\ZZ_{>0})$. 
For $H\in \AAA$,
let us define $\cone{\LAT}{H}$ to be the set 
\begin{gather*}
\Set{ \mu \in \LAT |  \mu_H > \frac{1}{2}\numof{\mu}  }.
\end{gather*} 
We define $\innercone{\LAT}$ and  $\innercone{\Lat}$  by
\begin{align*}
 \innercone{\LAT}&:= \LAT \setminus (\bigcup_{H\in \AAA} \cone{\LAT}{H})
=\Set{ \mu \in \LAT |  \mu_H \leq \frac{1}{2}\numof{\mu}  (\forall H \in \AAA)},\\
 \innercone{\Lat}&:=  \innercone{\LAT}\cap \Lat.
\end{align*}
\end{definition}
Roughly speaking, $\innercone{\LAT}$ consists of balanced elements
while $\cone{\LAT}{H}$ consists of elements 
such that $H$ monopolizes at least half of their multiplicities.

\begin{example}
\label{ex:a2}
Let $\AAA$ consist of three lines. 
In this case, 
\begin{align*}
\LAT&=\Set{(\mu_1,\mu_2,\mu_3)|\mu_i\in\NN},\\
\cone{\LAT}{1}&=\Set{(\mu_1,\mu_2,\mu_3)\in\LAT|\mu_1>\mu_2+\mu_3},\\
\cone{\LAT}{2}&=\Set{(\mu_1,\mu_2,\mu_3)\in\LAT|\mu_2>\mu_1+\mu_3},\\
\cone{\LAT}{3}&=\Set{(\mu_1,\mu_2,\mu_3)\in\LAT|\mu_3>\mu_1+\mu_2},
\intertext{and}
\innercone{\LAT}&=\Set{(\mu_1,\mu_2,\mu_3)\in\LAT|
\begin{array}{c}
\mu_1\leq \mu_2+\mu_3;\\
\mu_2\leq \mu_1+\mu_3;\\
\mu_3\leq \mu_1+\mu_2
\end{array}
}.
\end{align*}
By the result in Wakamiko \cite{W},
the exponents in this case can be described explicitly, 
and we have
\begin{align*}
\Delta(\mu)=
\begin{cases}
1&\text{if $\mu\in\innercone{\LAT}$ and $\numof{\mu}$ is odd,}\\
0& \text{if $\mu\in\innercone{\LAT}$ and $\numof{\mu}$ is even,}\\
2\mu_i-\numof{\mu} & \text{if $2\mu_i>\numof{\mu}$.}
\end{cases}
\end{align*}
Hence we have
$\innercone{\Lat}=\Set{\mu\in\innercone{\LAT}|\text{$\numof{\mu}$ is odd}}$.
\end{example}

For each $\mu\in \Lat$, 
there exist $\theta_\mu$ and $\theta'_\mu$ such that
 $\deg(\theta_\mu) < \deg(\theta'_\mu)$ and 
$\Set{\theta_\mu,\theta'_\mu}$ is a homogeneous basis
for $D(\AAA,\mu)$.
Since $\Delta(\mu) \neq 0$, $\theta_\mu$ is unique up to a nonzero scalar
for each $\mu\in \Lat$. 
Hence we can define 
a map $\theta: \Lat \to D(\AAA,0)=\Der{\kk}{S}$ by 
$\theta(\mu):=\theta_\mu$ (up to a scalar, 
or regard the image of $\theta$ as 
a one-dimensional vector space of $D(\AAA,0)$). 
\begin{definition}
Let us define $\cc{\Lat}$, 
$\fcc{\Lat}$ and $\icc{\Lat}$  
by
\begin{align*}
\cc{\Lat}  &:= \Set{\text{connected components of $\Lat$}},\\
\fcc{\Lat} &:= \Set{C \in \cc{ \Lat} | \numof{C} < \infty},\\
\icc{\Lat} &:= \Set{C \in \cc{ \Lat} | \numof{C} = \infty}, 
\end{align*}
where $\mu$ and $\nu$ are said to be \textit{connected} if there exists a path
 from $\mu$ to $\nu$ in the induced subgraph $\Lat$ of the Hasse graph.
For $C \in \cc{\Lat}$, $\mu \in C$ and $H \in \AAA$, 
define $\sectional{C}{\mu}{H}$ to be the set of $\nu \in C$
satisfying the following two conditions:
\begin{enumerate}
\item $\nu_{H'} = \mu_{H'}$ for each $H' \in \AAA\setminus \Set{H}$.
\item If $\nu\subset\kappa\subset\mu$ or $\mu\subset\kappa\subset\nu$, then $\kappa\in C$.
\end{enumerate}
\end{definition}
\begin{definition}
For $C\in \fcc{\Lat}$, we define $\mpts{C}$ by
\begin{align*}
\mpts{C}&:=
\Set{\mu \in C|\Delta(\mu) = \max\Set{ \Delta(\nu)| \nu\in C }}  
\intertext{and $\mpts{\Lat}$ by}
\mpts{\Lat}&:=\bigcup_{C\in\fcc{\Lat}} \mpts{C}.
\end{align*}
\end{definition}

\begin{example}
Let us consider the same $\AAA$ as Example \ref{ex:a2},
i.e., an arrangement consisting of three lines. 
In this case,
\begin{align*}
\fcc{\Lat} &=\Set{ \Set{\mu} | \mu \in \innercone{\Lat}},\\
\icc{\Lat} &=\Set{\cone{\LAT}{1},\cone{\LAT}{2},\cone{\LAT}{3}},\\
\cc{\Lat}  
&=\fcc{\Lat} \cup \Set{\cone{\LAT}{1},\cone{\LAT}{2},\cone{\LAT}{3}}.
\end{align*}
\end{example}

\begin{definition}
For a saturated chain
$\rho$ in $\LAT$,
i.e., a sequence 
$\rho=(\rho^{(0)},\ldots,\rho^{(k)})$ of elements in $\LAT$
satisfying $\rho^{(i)} \coveredby \rho^{(i+1)}$,
we define $\downalpha{\rho}$ by
\begin{gather*}
\downalpha{\rho}=
\prod_{\text{$i$: $\Delta(\rho^{(i)}) > \Delta(\rho^{(i+1)})$}}  
\alpha^{(i)} ,
\end{gather*}
where $\alpha^{(i)}=\alpha_{H}$ such that
$\rho^{(i)}_{H} + 1 = \rho^{(i+1)}_{H}$.
\end{definition}

\section{Main Results}\label{section:main}
In this section we state the main results. First 
let us give three theorems which 
show the structure of $\Lat$.

\begin{theorem}\label{theorem:misc}
We have the following:
\begin{enumerate}
\item For each  $C \in \fcc{\Lat}$, it holds that $ C \subset \innercone{\LAT}'$.
Moreover,  $\bigcup_{C \in \fcc{\Lat}} C = \innercone{\Lat}$.
\item $\icc{\Lat}=\Set{\cone{\LAT}{H} | H \in \AAA}$.
\item Any maximal connected component of $\LAT\setminus \Lat =\Delta^{-1}(\Set{0})$ consists
      of one point. 
\end{enumerate}
\end{theorem}

\begin{theorem}\label{theorem:shape}
Let $C\in \fcc{\Lat}$ and $\mu\in\mpts{C}$.
Then
\begin{gather*}
C=\disk{\mu}{\Delta(\mu)},
\end{gather*}
and, for $\nu \in C$, 
\begin{gather*}
\Delta(\nu)=\Delta(\mu)-\distance{\mu}{\nu}.
\end{gather*}
In particular, 
for $C \in \fcc{\Lat}$,  
$\mpts{C}$ consists of one point.
\end{theorem}

\begin{cor}\label{cor:distance}
For $C \in \fcc{\Lat},\ \mu\in\mpts{C}$ and $\nu\in\LAT$ satisfying 
$\distance{\mu}{\nu}<\Delta(\mu)+2$,
\begin{gather*}
\Delta(\nu)=|\Delta(\mu)-\distance{\mu}{\nu}|.
\end{gather*}
\end{cor}

The following result implies the independency of ``low-degree'' bases. 

\begin{theorem}\label{theorem:independency}
Let $C, C'\in \fcc{\Lat}$ such that $\distance{C}{C'}=2$. 
If $\mu\in C$ and $\mu' \in C'$,  
then $\Set{\theta_\mu,\theta_{\mu'}}$ is $S$-linearly independent. 
Moreover, 
if $C \in \cc{\Lat}$ and $\mu,\mu'\in C$, 
then $\Set{\theta_\mu,\theta_{\mu'}}$ is $S$-linearly dependent.
\end{theorem}

The theorems above imply the following three corollaries, 
which enable us to construct the basis for $D(\AAA,\mu)$ effectively. 

\begin{cor}\label{cor:criterion:I}
Let $N \subset \innercone{\LAT}$ 
such that
$\innercone{\LAT} \setminus N$ does not have any connected component
 whose size is larger than $1$, and let 
 $\vartheta:N \to D(\AAA,0)$ such that
$\vartheta_\mu \in D(\AAA,\mu)$ and 
$\deg \vartheta_\mu < \frac{\numof{\mu}}{2}$.
Then the following are equivalent:
\begin{itemize} 
\item $\Set{\vartheta_\mu,\vartheta_{\nu}}$ is $S$-linearly independent
if 
\begin{gather*}
\min{\Set{\distance{\mu'}{\nu'}|\substack{
\text{$\mu$ and $\mu'$ are in the same connected component in $N$}\\
\text{$\nu$ and $\nu'$ are in the same connected component in $N$}}}}=2.
\end{gather*}
\item $N = \innercone{\Lat}$ .
\end{itemize}
\end{cor}

\begin{cor}\label{cor:criterion:II}
Let $N \subset \innercone{\LAT}$
and $\vartheta:N \to D(\AAA,0)$ such that
$\vartheta_\mu \in D(\AAA,\mu)$,  
$\deg \vartheta_\mu < \frac{\numof{\mu}}{2}$
and $\Delta'(\mu)=\numof{\mu}-2\deg \vartheta_\mu>0$.
Assume that 
$\disk{\mu}{\Delta'(\mu)}$ and  $\disk{\nu}{\Delta'(\nu)}$ are disjoint
 for $\mu \neq \nu \in N$, and that
$\innercone{\LAT} \setminus \bigcup_{\mu \in N} \disk{\mu}{\Delta'(\mu)}$ 
 has no connected components
 whose size is larger than $1$.
Then the following are equivalent:
\begin{itemize}
\item  $\Set{\vartheta_\mu,\vartheta_{\nu}}$ are $S$-linearly independent
if $\Delta'(\mu)+\Delta'(\nu) = \distance{\mu}{\nu}$.
\item $N=\mpts{\Lat}$ and $\vartheta_\mu=\theta_\mu$ 
for each $\mu\in\mpts{\Lat}$. 
\end{itemize}
\end{cor}

\begin{cor}\label{cor:criterion:III}
Let $N=\Set{ \mu\in \innercone{\LAT} | \text{$\numof{\mu}$ is odd.} }$
and $\vartheta:N \to D(\AAA,0)$ such that
$\vartheta_\mu \in D(\AAA,\mu)$ and $\deg \theta_\mu < \frac{\numof{\mu}}{2}$.
Define the equivalence relation  $\sim$  generated by
\begin{gather*}
\mu \sim \nu
\iff
\text{ $\Set{\vartheta_\mu,\vartheta_\nu}$ is $S$-linearly dependent
 and $\distance{\mu}{\nu}=2$}.
\end{gather*}
Then the following are equivalent for $\mu,\nu\in N$:
\begin{itemize}
\item $\mu\sim\nu$ .
\item $\mu,\nu\in C \in \fcc{\Lat}$. 
\end{itemize}
\end{cor}

\begin{remark}
 In Corollaries \ref{cor:criterion:I}, \ref{cor:criterion:II} and  \ref{cor:criterion:III},
we do not require the condition $\deg(\vartheta_\mu) = \deg(\theta_\mu)$.
\end{remark}

Finally we state the theorems which describe the behavior 
of the basis near, or between the centers of connected balls.

\begin{theorem}\label{theorem:basis}
Assume that $\mu,\nu\in \Lat$ belong to distinct connected components 
and satisfy $\Delta(\mu)+\Delta(\nu)=\distance{\mu}{\nu}$.
Let $\kappa\in \LAT$ such that $\mu \wedge \nu \subset \kappa \subset \mu \vee \nu$,
and
\begin{align*}
\alpha_{\mu,\kappa}&=\prod_{H\in \AAA} \alpha_H^{\max\Set{\kappa_H-\mu_H,0}},\\
\alpha_{\nu,\kappa}&=\prod_{H\in \AAA} \alpha_H^{\max\Set{\kappa_H-\nu_H,0}}.
\end{align*}
Then $\Set{\alpha_{\mu,\kappa}\theta_{\mu}, \alpha_{\nu,\kappa}\theta_\nu}$
is a homogeneous basis for $D(\AAA,\kappa)$.
\end{theorem}

\begin{cor}
For each $\mu \in \innercone{\LAT}$,
we can construct a homogeneous basis for $D(\AAA,\mu)$
from the restricted map $\theta|_{\mpts{\Lat}}$.
\end{cor}

\section{Proofs of Main Results}\label{section:proof}
In this section, we prove the main results.
To prove them,
first we recall a result about  hyperplane arrangements and
derivation modules. The following is the two-dimensional version 
of the famous Saito's criterion, which is very useful to find the 
basis for $D(\AAA,m)$. See Theorem 8 in \cite{Z} and Theorem 4.19 in \cite{OT} 
for the proof. 

\begin{theorem}[Saito's criterion]
Let $(\AAA,\mu)$ be a $2$-multiarrangement and $\theta_1,\theta_2 \in D(\AAA,\mu)$. 
Then $\{\theta_1,\theta_2\}$ forms a basis for $D(\AAA,\mu)$ if and only if 
$\{\theta_1,\theta_2\}$ is independent and $\deg(\theta_1)+\deg(\theta_2)=|\mu|$.
\label{Saito}
\end{theorem}

\subsection{Proofs of Theorems \ref{theorem:misc} and \ref{theorem:shape}}

\begin{lemma}
\label{lemma:s}
If $\mu,\nu \in \LAT$ and $\mu \coveredby \nu$,
then $|\Delta(\mu)-\Delta(\nu)| =1$.
\end{lemma}

\begin{proof}
It follows from the fact that $D(\AAA,\mu)\supset D(\AAA,\nu)$ 
and Saito's criterion. 
\end{proof}

\begin{lemma}
\label{lemma:bas}
Assume that $\mu,\nu \in \LAT'$ and $\mu \coveredby \nu$ 
with 
$\nu_H=\mu_H+1$ for some $H \in \AAA$. Then 
\begin{gather*}
\theta_\nu=
\begin{cases}
\alpha_H \theta_\mu & \text{if $\Delta(\mu) > \Delta(\nu)$,} \\
\theta_\mu  & \text{if $\Delta(\mu) < \Delta(\nu)$.} 
\end{cases}
\end{gather*}
\end{lemma}

\begin{proof} 
Fix a homogeneous basis 
$\Set{\theta_\mu, \theta'}$
 for $D(\AAA,\mu)$,
where $\deg(\theta_\mu) < \deg(\theta') $.
If $\Delta(\mu) > \Delta(\nu)$, then 
Saito's criterion implies $\theta_\mu \not\in D(\AAA,\nu)$.
Since $\alpha_H\theta_{\mu}\in D(\AAA,\nu)$, Lemma \ref{lemma:s}
implies $\alpha_H \theta_\mu$ is a part of a 
homogeneous basis for $D(\AAA,\nu)$.
Hence we may assume that 
$\Set{\alpha_H\theta_\mu, \theta''}$ is a basis for
 $D(\AAA,\nu)$.
If $\Delta(\mu) < \Delta(\nu)$, then 
$\theta_\mu \in D(\AAA,\nu)$,
which completes the proof.
\end{proof}

\begin{cor}
\label{lemma:pathbasis}
Let 
$\mu,\nu \in C \in \cc{\Lat}$ with $\mu \subset \nu$, and $\rho$ be a saturated chain
from $\mu$ to $\nu$.
Then $\theta_\nu=\downalpha{\rho} \theta_\mu$.
\end{cor}

\begin{proof}
Apply Lemma \ref{lemma:bas} repeatedly.
\end{proof}

\begin{lemma}
\label{lemma:unimodal}
Let $C \in \cc{\Lat}$, $\mu \in C$, and $H \in \AAA$.
If $\numof{\sectional{C}{\mu}{H}}<\infty$, then
$\Delta |_{\sectional{C}{\mu}{H}} $ is unimodal, or equivalently,
there exists a unique element $\kappa\in \sectional{C}{\mu}{H}$ such that
\begin{align*}
&\text{$\Delta(\nu')\leq \Delta(\nu)$ for $\nu'\subset \nu \subset \kappa$ or $\kappa \subset \nu \subset \nu'.$ }
\end{align*}
If $\numof{\sectional{C}{\mu}{H}}=\infty$, then
$\Delta |_{\sectional{C}{\mu}{H}} $ is monotonic, or equivalently,
\begin{align*}
&\text{$\Delta(\nu')\leq \Delta(\nu)$ for $\nu'\subset \nu$. }
\end{align*}
\end{lemma}

\begin{proof}
Let
$\nu,\nu',\nu'' \in \sectional{C}{\mu}{H}$
satisfy  $\nu\coveredby\nu'\coveredby\nu''$.
Assume that 
$\Delta(\nu) > \Delta(\nu')<\Delta(\nu'')$.
By Lemma \ref{lemma:bas}, we may choose a basis 
$\Set{\alpha_H\theta_{\nu}, \theta'}$ for $D(\AAA,\nu')$ such that 
$\Set{\alpha_H\theta_{\nu}, \alpha_H \theta'}$ is a basis for 
$D(\AAA,\nu'')$. 
Hence 
$\alpha_H \theta_{\nu}(\alpha_H) \in \ideal{\alpha_H^{\nu''_{H}}}=\ideal{\alpha_H^{\nu_{H}+2}}$ and 
$\theta_{\nu}(\alpha_H) \in \ideal{ \alpha_H^{\nu_{H}+1} }$. 
Then $\theta_{\nu} \in D(\AAA,\nu')$, which is a contradiction. 
Since it follows from Lemma \ref{lemma:s} that
 $\min\Set{ \Delta(\mu')| \mu' \in \sectional{C}{\mu}{H}}=1$,
we have the lemma.
\end{proof}

\begin{definition}
For $H \in \AAA$, $C \in \fcc{\Lat}$ and $\mu \in C$, we may choose, by Lemma 
\ref{lemma:unimodal}, the unique element 
$\kappa \in C_{\mu,H}$ such that 
$\Delta(\kappa) \ge \Delta(\mu')$ for any $\mu' \in C_{\mu,H}$. We call this 
$\kappa$ the \defit{peak element with respect to $\sectional{C}{\mu}{H}$}.
\end{definition}

\begin{cor}
Let $C \in \fcc{\Lat}$, $\mu \in C$ and $H \in \AAA$. 
Let $\kappa \in C$ be the peak element with respect to $\sectional{C}{\mu}{H}$. 
Then, for $\mu' \in \sectional{C}{\mu}{H}$, 
\begin{gather*}
\theta_{\mu'}=
\begin{cases}
\theta_\kappa & (\mu' \subset \kappa),\\
\alpha_H^{\numof{\mu'} - \numof{\kappa}} \theta_\kappa & ( \kappa \subset \mu').
\end{cases}
\end{gather*}
\end{cor}

\begin{proof}
Apply Lemmas \ref{lemma:bas} and \ref{lemma:unimodal}.
\end{proof}

\begin{lemma}\label{lemma:latA:Detail}
Let $C\in \cc{\Lat}$, and $\kappa, \mu,\mu',\nu \in \LAT$. 
Assume that 
$\kappa \coveredby \mu \coveredby \nu$,
$\kappa \coveredby \mu' \coveredby \nu$, and 
$\mu \neq \mu'$.
\begin{enumerate}
\item\label{lemma:latA:case:I}
Assume that $\kappa, \mu , \mu' \in C$. Then
\begin{align*}
\text{$\Delta(\kappa)>\Delta(\mu)$ and  $\Delta(\kappa)>\Delta(\mu')$}
&\implies 
\text{$\Delta(\mu)>\Delta(\nu)$ and  $ \Delta(\mu')>\Delta(\nu);$}\\
\text{$\Delta(\kappa)<\Delta(\mu)$ and  $ \Delta(\kappa)<\Delta(\mu')$}
&\implies 
\text{$\Delta(\mu)<\Delta(\nu)$ and  $ \Delta(\mu')<\Delta(\nu) ;$ and}\\
\text{$\Delta(\kappa)<\Delta(\mu)$ and  $ \Delta(\kappa)>\Delta(\mu')$}
&\implies 
\text{$\Delta(\mu)>\Delta(\nu)$ and  $ \Delta(\mu')<\Delta(\nu)$} .
\end{align*}
\item\label{lemma:latA:case:II}
Assume that $ \mu , \mu',\nu \in C$. Then
\begin{align*}
\text{$\Delta(\mu)>\Delta(\nu)$ and  $ \Delta(\mu')>\Delta(\nu)$}
&\implies 
\text{$\Delta(\kappa)>\Delta(\mu)$ and  $ \Delta(\kappa)>\Delta(\mu');$} \\
\text{$\Delta(\mu)<\Delta(\nu)$ and  $ \Delta(\mu')<\Delta(\nu)$}
&\implies 
\text{$\Delta(\kappa)<\Delta(\mu)$ and  $ \Delta(\kappa)<\Delta(\mu');$ and},\\
\text{$\Delta(\mu)<\Delta(\nu)$ and  $ \Delta(\mu')>\Delta(\nu)$}
&\implies 
\text{$\Delta(\kappa)>\Delta(\mu)$ and  $ \Delta(\kappa)<\Delta(\mu')$} .
\end{align*}
\item\label{lemma:latA:case:III}
Assume that $ \kappa,\mu ,\nu \in C$. Then
\begin{align*}
\Delta(\kappa) > \Delta(\mu) > \Delta(\nu)
&\implies 
\Delta(\kappa) > \Delta(\mu') > \Delta(\nu);\\
\Delta(\kappa) < \Delta(\mu) < \Delta(\nu)
&\implies 
\Delta(\kappa) < \Delta(\mu') < \Delta(\nu);\\
\Delta(\kappa) < \Delta(\mu) > \Delta(\nu)
&\implies 
\Delta(\kappa) > \Delta(\mu') < \Delta(\nu);\text{ and}\\
\Delta(\kappa) > \Delta(\mu) < \Delta(\nu)
&\implies 
\Delta(\kappa) < \Delta(\mu') > \Delta(\nu).
\end{align*}
\end{enumerate}
\end{lemma}

\begin{proof}
$(\ref{lemma:latA:case:I})$ 
Assume that 
$\kappa_H +1= \mu_H$ and $\kappa_{H'} +1= \mu'_{H'}$ for 
some $H \neq H' \in \AAA$. 
Since $\nu = \mu \vee \mu'$, $\mu_{H'} +1= \nu_{H'}$ and 
$\mu'_{H}+1= \nu_{H}$. 
First we consider the case when 
$\Delta(\kappa)>\Delta(\mu)$ and $\Delta(\kappa)>\Delta(\mu')$. Then 
$\Delta(\mu)=\Delta(\mu')$. 
It follows from Lemma \ref{lemma:bas} that
$\theta_\mu    = \alpha_H \theta_\kappa$, $\theta_{\mu'} = \alpha_{H'} \theta_\kappa$.
If $\Delta(\mu)=\Delta(\mu')<\Delta(\nu)$, then
$\Delta(\nu)>0$, i.e., $\nu \in \Lat$. Then 
Lemma \ref{lemma:bas} implies that
\begin{align*}
\alpha_{H'} \theta_\kappa= \theta_{\mu'}=
\theta_\nu= \theta_\mu = \alpha_H \theta_\kappa,
\end{align*}
which is a contradiction.

Next we consider the case when 
$\Delta(\kappa)<\Delta(\mu)$ and $\Delta(\kappa)<\Delta(\mu')$. Then 
$\Delta(\mu)=\Delta(\mu')$ and
 $\Delta(\kappa)\leq \Delta(\nu)$.  
Hence $\nu\in C$.
It follows from Lemma \ref{lemma:bas} that
$\theta_\mu    =  \theta_\kappa$,
$\theta_{\mu'} =  \theta_\kappa$.
If $\Delta(\mu)=\Delta(\mu')>\Delta(\mu)$, then
Lemma \ref{lemma:bas} implies that
\begin{align*}
\alpha_{H} \theta_\kappa=\alpha_{H}\theta_{\mu'} = \theta_\nu&=
 \alpha_{H'}\theta_\mu = \alpha_{H'} \theta_\kappa,
\end{align*}
which is a contradiction.

Finally we consider the case when 
$\Delta(\kappa)<\Delta(\mu)$ and $\Delta(\kappa)>\Delta(\mu')$. Then 
$\Delta(\mu)-1=\Delta(\mu')+1 =\Delta(\kappa)$.
Hence $\Delta(\nu) =\Delta(\mu)-1= \Delta(\mu')+1=\Delta(\kappa)$.

The same argument is valid for $(\ref{lemma:latA:case:II})$ and
 $(\ref{lemma:latA:case:III})$, which completes the proof.
\end{proof}

\begin{remark}
In cases $(\ref{lemma:latA:case:I})$, $(\ref{lemma:latA:case:II})$ and
 $(\ref{lemma:latA:case:III})$
in  Lemma \ref{lemma:latA:Detail}, 
$\Delta(\nu)=0$, $\Delta(\kappa)=0$ and $\Delta(\mu')=0$ may happen, respectively.  
\end{remark}

\begin{lemma}
\label{lemma:latA}
Let $\mu,\mu'\in \LAT$ such that $\numof{\mu}=\numof{\mu'}$,
 $\mu\neq\mu'$ and $\distance{\mu}{\mu'}=2$.
Then the following are equivalent:
\begin{enumerate}
\item 
\label{lemma:latA:1}
At least three of $\Set{\mu\wedge\mu',\mu,\mu',\mu\vee\mu'}$
are in the same connected component $C\in\cc{\Lat}$.
\item 
\label{lemma:latA:1.5}
At least three of $\Delta(\mu\wedge\mu')$, $\Delta(\mu)$,
$\Delta(\mu')$ and $\Delta(\mu\vee\mu')$
are positive.
\item 
\label{lemma:latA:2}
$\Delta(\mu\vee\mu')-\Delta(\mu') =\Delta(\mu)-\Delta(\mu\wedge\mu')$.
\item 
\label{lemma:latA:3}
$\Delta(\mu\vee\mu')-\Delta(\mu) =\Delta(\mu')-\Delta(\mu\wedge\mu')$.
\end{enumerate}
\end{lemma}
\begin{proof}
By the assumption, 
$|\mu\vee\mu'|-1=|\mu|=|\mu'|=|\mu\wedge\mu'|+1$.
It follows from Lemma \ref{lemma:s} that
$|\Delta(\mu\vee\mu')-\Delta(\mu')|=| \Delta(\mu)-\Delta(\mu\wedge\mu')| =|\Delta(\mu\vee\mu')-\Delta(\mu) |=|\Delta(\mu')-\Delta(\mu\wedge\mu')|=1$.
It is clear that
Conditions (\ref{lemma:latA:2}) and (\ref{lemma:latA:3}) are equivalent.
It is also clear that
Conditions (\ref{lemma:latA:1}) and (\ref{lemma:latA:1.5}) are equivalent.
It follows from Lemma \ref{lemma:latA:Detail}
that Condition (\ref{lemma:latA:1}) implies Condition (\ref{lemma:latA:2}).
Now we show that Condition (\ref{lemma:latA:2}) implies Condition (\ref{lemma:latA:1.5}).
If two of $\Delta(\mu\vee\mu')$, $\Delta(\mu')$, $\Delta(\mu)$
and $\Delta(\mu\wedge\mu')$ are zero, 
then Lemma \ref{lemma:s} shows that we have one of the following two:
\begin{itemize}
\item
$\Delta(\mu\vee\mu')=\Delta(\mu\wedge\mu')=1$ and
     $\Delta(\mu')=\Delta(\mu)=0$; or
\item 
$\Delta(\mu\vee\mu')=\Delta(\mu\wedge\mu')=0$ and 
$\Delta(\mu')=\Delta(\mu)=1$.
\end{itemize}
Both of them contradicts Condition (\ref{lemma:latA:2}).
\end{proof}

\begin{lemma}
\label{lemma:shape}
For $\mu \in C \in \cc{\Lat}$, define $X_\mu$ by
$X_\mu := \bigcup_{H\in \AAA} \sectional{C}{\mu}{H}$. 
If $\mu$ satisfies 
$\Delta(\mu) = \max \Set{\Delta(\nu)|\nu\in X_\mu}$,
then
\begin{gather*}
\Delta(\kappa)=\Delta(\mu) - \distance{\kappa}{\mu}
\end{gather*}
 for 
$\kappa\in \LAT$ with $\distance{\kappa}{\mu} \leq \Delta(\mu)$. 
In particular, $ C $ is the ball $\disk{\mu}{\Delta(\mu)}$.
\end{lemma}

\begin{proof}

If $\Delta(\mu)=1$ then there is nothing to prove. 
Assume that $\Delta(\mu)>1$. Since $\mu$ satisfies 
$\Delta(\mu) = \max \Set{\Delta(\nu)|\nu\in X_\mu}$,
it follows from Lemma \ref{lemma:unimodal}
that $\Delta(\kappa)=\Delta(\mu) - \distance{\kappa}{\mu}$
for $\kappa \in X_\mu$.
In particular, we have the lemma for $\distance{\kappa}{\mu}=1$.

Now we prove the lemma by the induction on $\distance{\kappa}{\mu}$.
Let $\distance{\kappa}{\mu}>1$. 
By the previous paragraph it suffices to show 
when $\kappa \not \in X_\mu$. 
In this case, there exists $H'\neq H''$ such that
$\mu_{H'} \neq \kappa_{H'}$
and 
$\mu_{H''} \neq \kappa_{H''}$.
Let us define $\kappa'$, $\kappa''$, $\kappa'''$ by
\begin{align*}
\kappa'_H&=
\begin{cases}
\kappa_H    &\text{if $H\neq H'$,}\\
\kappa_{H'}-1 &\text{if $H= H'$ and $\kappa_{H'} > \mu_{H'}$,}\\
\kappa_{H'}+1 &\text{if $H=H'$  and $\kappa_{H'} < \mu_{H'}$,}
\end{cases}\\
\kappa''_H&=
\begin{cases}
\kappa_H    &\text{if $H\neq H''$,}\\
\kappa_{H''}-1 &\text{if $H= H''$ and $\kappa_{H''} > \mu_{H''}$,}\\
\kappa_{H''}+1 &\text{if $H=H''$  and $\kappa_{H''} < \mu_{H''}$,}
\end{cases}\\
\kappa'''_H&=
\begin{cases}
\kappa_H    &\text{if $H'\neq H\neq H''$,}\\
\kappa'_{H'} &\text{if $H= H'$,}\\
\kappa''_{H''} &\text{if $H=H''$.}
\end{cases}
\end{align*}
Then 
$\distance{\kappa}{\mu}-1=\distance{\kappa'}{\mu}=\distance{\kappa''}{\mu}=\distance{\kappa'''}{\mu}+1$.
By the induction hypothesis, 
 $\Delta(\kappa')=\Delta(\kappa'')=\Delta(\kappa''')-1=\Delta(\mu) - \distance{\kappa}{\mu}+1>0$.
It follows from Lemma \ref{lemma:latA} that
$\Delta(\kappa)-\Delta(\kappa')=\Delta(\kappa'') - \Delta(\kappa''')=-1$.
Hence $\Delta(\kappa)=\Delta(\mu) - \distance{\kappa}{\mu}$.
\end{proof}

\begin{proof}[Proof of Theorem \ref{theorem:shape}]
Let $C\in \fcc{\Lat}$ and $\mu \in \mpts{C}$.
Then it follows from Lemma \ref{lemma:unimodal} that
$\Delta |_{\sectional{C}{\mu}{H}}$ is unimodal for all $H \in \AAA$.
Hence Lemma \ref{lemma:shape} completes the proof.
\end{proof}

\begin{lemma}
\label{lemma:preservemax}
Let $H\in\AAA$ and $\mu,\mu' \in C\in \icc{\Lat}$ 
satisfy $\mu \coveredby \mu'$ with 
$\mu_H + 1 = \mu'_H $.
If $\numof{\sectional{C}{\mu}{H'}}<\infty$ for some $H'\in \AAA\setminus\Set{H}$, 
then $\numof{\sectional{C}{\mu'}{H'}}<\infty$.
Moreover, 
for $H'\in \AAA\setminus\Set{H}$,
$\mu$ is the peak element with respect to $\sectional{C}{\mu}{H'}$
if and only if
$\mu'$ is the peak element with respect to $\sectional{C}{\mu'}{H'}$. 
\end{lemma}

\begin{proof}

First consider the case when $\numof{\sectional{C}{\mu}{H'}}=1$.
In this case, $\Delta(\mu)=1$ and $\Delta(\mu')=2$.
Define $\mu^{(1)}$ by $ \mu  \coveredby \mu^{(1)}$ with 
$\mu_{H'}+1 =\mu^{(1)}_{H'}$.
Then $\mu^{(1)}\vee\mu' \in\sectional{C}{\mu'}{H'}$.
By the assumption $\Delta(\mu^{(1)})=0$. So Lemma \ref{lemma:latA} implies
$\Delta(\mu^{(1)}\vee\mu')=1$.
Define $\mu^{\prime(-1)}$ by $ \mu^{\prime(-1)}  \coveredby \mu'$ with 
$\mu^{\prime(-1)}_{H'}+1 =\mu'_{H'}$.
Then 
$\mu^{\prime(-1)}\in \sectional{C}{\mu'}{H'}$.
By the assumption $\Delta(\mu\wedge\mu^{\prime(-1)})=0$. So Lemma \ref{lemma:latA} 
implies
$\Delta(\mu^{\prime(-1)})=1$.
Since $\mu^{\prime(-1)} \coveredby\mu' \coveredby\mu^{(1)}\vee\mu' $
and $\Delta(\mu^{(1)}\vee\mu')<\Delta(\mu')>\Delta(\mu^{\prime(-1)})$,
by Lemma \ref{lemma:unimodal}, 
 $\mu'$ is the peak element with respect to
 $\sectional{C}{\mu'}{H'}$, and  $\numof{\sectional{C}{\mu'}{H'}}=3 <\infty$.

Next consider the case when $\numof{\sectional{C}{\mu}{H'}}>1$.
Let $\mu^{(0)}$ be the peak element with respect to $\sectional{C}{\mu}{H'}$,
and 
\begin{align*}
\sectional{C}{\mu}{H'}=\Set{
\mu^{(i)} \mid 
\cdots \coveredby  \mu^{(-1)} \coveredby \mu^{(0)}\coveredby\mu^{(1)}\coveredby \cdots} .
\end{align*}
Then
$\Delta(\mu^{(-i)})=\Delta(\mu^{(i)})$.
Let us define $\mu^{\prime(i)}$ by 
 $\mu^{(i)} \coveredby \mu^{\prime(i)}$ and 
$\mu^{(i)}_H + 1 = \mu^{\prime(i)}_H $.
If $\mu=\mu^{(j)}$, then $\mu'=\mu^{\prime(j)}$.
By direct calculation, we have
 $ \mu^{\prime(i)} \vee \mu^{(i+1)}=\mu^{\prime(i+1)}$
and  $ \mu^{\prime(i)} \wedge \mu^{(i+1)}=\mu^{(i)}$.
For $i<0$,  $\Delta(\mu^{(i+1)})>\Delta(\mu^{(i)})>0$.
Hence $\Delta(\mu^{\prime(i+1)})>0$.
It follows from Lemma \ref{lemma:latA} that
\begin{align*}
\Delta(\mu^{\prime(i+1)})-\Delta(\mu^{\prime(i)})=\Delta(\mu^{(i+1)})-\Delta(\mu^{(i)})=1.
\end{align*}
On the other hand, for $i>0$, the same argument implies that 
\begin{align*}
\Delta(\mu^{\prime(i-1)})-\Delta(\mu^{\prime(i)})=\Delta(\mu^{(i-1)})-\Delta(\mu^{(i)})=1.
\end{align*}
Hence, by Lemma  \ref{lemma:unimodal},
 $\mu^{\prime(k)}$ is the peak element with respect to
 $\sectional{C}{\mu'}{H'}$, and  $\numof{\sectional{C}{\mu'}{H'}} <\infty$. 
The same proof is valid if $\mu$ is replaced by $\mu'$. 
\end{proof}

\begin{lemma}
\label{lemma:infty}
Let $C\in\cc{\Lat}$. If there exists $\mu\in C$ satisfying
$\numof{\sectional{C}{\mu}{H}}<\infty$ for any $H\in\AAA$,
then $C\in\fcc{\Lat}$.
Hence, for $\mu\in C\in\icc{\Lat}$,
there exists $H \in \AAA$  such that
$\numof{\sectional{C}{\mu}{H}}=\infty$.
\end{lemma}

\begin{proof}
For $H \in \AAA$, $C\in\cc{\Lat}$ and $\mu\in C$, 
define $m_{\mu, H}$ and $\BBB_{\mu}$ by
$m_{\mu, H} :=\max\Set{\Delta(\mu')| \mu' \in \sectional{C}{\mu}{H}}$
and $\BBB_{\mu}:=\Set{H \in \AAA|\Delta(\mu)=m_{\mu,H}}$. 
Assume that 
$\numof{\sectional{C}{\mu}{H}}<\infty$ for all $H\in\AAA$.
Let us construct $\nu$ as follows:
\begin{enumerate}
\item Let $\nu$ be $\mu$.
\item Repeat the following until $\AAA = \BBB_{\nu}$ :
\begin{enumerate}
\item Choose $H_0 \in \AAA \setminus \BBB_{\nu}$ and the peak
 element 
$\nu'$ with respect to $\sectional{C}{\nu}{H_0}$. 
\item Let $\nu$ be $\nu'$.
\end{enumerate}
\end{enumerate}
By the assumption and Lemma \ref{lemma:preservemax}, 
$\numof{\sectional{C}{\nu'}{H}}<\infty$ for all $H\in\AAA$ and 
$\Delta(\nu') = m_{\nu',H}$ for all $H \in \BBB_{\nu}$.
Hence, by Lemma \ref{lemma:preservemax},
$\BBB_{\nu'}=\BBB_{\nu} \cup \Set{H_0}$.
Since $\numof{\AAA} < \infty$, we can always find $\nu \in C$ such that
$\Delta(\nu) = m_{\nu,H}$ for all $H\in\AAA$.
Hence Lemma \ref{lemma:shape} implies that 
$C \in \fcc{\Lat}$. 
\end{proof}

\begin{lemma}\label{lemma:outsidecone}
 $\icc{\Lat}=\Set{ \cone{\LAT}{H} | H\in\AAA }$.
\end{lemma}
\begin{proof}
Lemma \ref{lemma:infty} implies that, 
for $\mu \in C \in \icc{\Lat}$, 
there exists $H$  such that
$\numof{\sectional{C}{\mu}{H}}=\infty$.
Hence if $\nu\in \LAT$ satisfies
\begin{gather*}
\nu_{H'}=
\begin{cases}
\mu_{H} + \numof{\mu} & (H=H'),\\
\mu_{H'} & (H\neq H'),
\end{cases}
\end{gather*}
then $\nu \in \sectional{C}{\mu}{H}$.
By definition, $\nu \in \cone{\LAT}{H}$. 
Since $\mu$ and $\nu$ belong to the same component $C$,
$\mu$ is also in $\cone{\LAT}{H}$.
On the other hand, $\cone{\LAT}{H} \in \icc{\Lat}$.
Since $\cone{\LAT}{H}$ is connected, $C=\cone{\LAT}{H}$. 
\end{proof}

\begin{proof}[Proof of Theorem \ref{theorem:misc}] 
Apply Lemma \ref{lemma:s} and \ref{lemma:outsidecone}.
\end{proof}

\subsection{Proof of Theorem \ref{theorem:independency}}
In this subsection we prove Theorem \ref{theorem:independency}. Roughly speaking, 
the proof is based on the observation of $\theta_{\mu}$ for $\mu$ in some finite balls in Theorem 
\ref{theorem:shape}.

\begin{lemma}\label{lemma:depincI}
Let $C\in \fcc{\Lat}$, $\kappa\in C$
and $\mu \in \mpts{C}$.
Then we can construct $\theta_\mu$ from $\theta_\kappa$, and vice versa.
\end{lemma}

\begin{proof}
By Theorem \ref{theorem:shape},  $\mu \wedge \kappa \in C$.
It follows from  Lemma \ref{lemma:pathbasis} that
\begin{align*}
\theta_\mu &=\downalpha{\rho} \theta_{\mu \wedge \kappa},\\
\theta_\kappa &=\downalpha{\rho'} \theta_{\mu \wedge \kappa}
\end{align*}
for some saturated chains  $\rho$ and $\rho'$.
Hence we have
\begin{align*}
\theta_\mu &= \frac{\downalpha{\rho}}{\downalpha{\rho'}}\theta_\kappa,\\
\theta_\kappa &= \frac{\downalpha{\rho'}}{\downalpha{\rho}}\theta_\mu.
\end{align*}
\end{proof}

\begin{lemma}\label{lemma:depinc}
Let $C\in\fcc{\Lat}$ and $\mu,\nu\in C$.
Then $\Set{\theta_\mu,\theta_\nu}$ is $S$-linearly dependent.
\end{lemma}

\begin{proof}
The lemma  follows from Lemma \ref{lemma:depincI}.
\end{proof}

\begin{lemma}\label{lemma:independency}
Let $\mu,\nu \in \Lat$ satisfy $\distance{\mu}{\nu}=2$.
If $\Delta(\kappa)=0$ for all $\kappa\in \LAT$ such that
 $\distance{\mu}{\kappa}=\distance{\nu}{\kappa}=1$, then
$\Set{\theta_\mu, \theta_\nu}$ is $S$-linearly independent.
\end{lemma}

\begin{proof}
First assume that $\mu_H+2 = \nu_H$ for some $H \in \AAA$ and
 $\mu_{H''}=\nu_{H''}$ for $H'' \in \AAA\setminus\Set{H}$.
Let $\kappa \in \LAT$ be the element such that 
$\mu \coveredby \kappa \coveredby \nu$. 
Since $\Delta(\kappa)=0$, 
$\theta_\mu \not\in D(\AAA,\kappa)$. 
Hence $\alpha_H \theta_\mu \in D(\AAA,\kappa)$ and is a 
part of basis. 
Let $\Set{\alpha_H \theta_\mu, \theta' }$ be a basis for 
the $S$-module $D(\AAA,\kappa)$.
Since $D(\AAA,\nu) \subset D(\AAA,\kappa)$,
$\theta_\nu= a \alpha_H \theta_\mu + b \theta' $ for some $a,b \in \kk$.
If $\Set{\theta_\mu, \theta_\nu }$ is $S$-linearly dependent, then
$b=0$, i.e., $\theta_\nu = a \alpha_H \theta_\mu$.
Since $\alpha_H \theta_\mu \in D(\AAA,\nu)$,
$\alpha_H \theta_\mu (\alpha_H) \in \ideal{\alpha_H^{\nu_H}}=\ideal{\alpha_H^{\kappa_H+1}}$.
Hence $\theta_\mu (\alpha_H) \in \ideal{\alpha_H^{\kappa_H}}$ and 
$\theta_\mu \in D(\AAA,\kappa)$, 
which is a contradiction.

Next assume that 
$\mu_H +1=\nu_H$ and $\mu_{H'}+1= \nu_{H'}$ for some 
$H,H' \in \AAA$ 
and  $\mu_{H''}=\nu_{H''}$ for $H'' \in \AAA \setminus \Set{H, H'}$.
Let $\kappa \in \LAT$ be the element such that $\kappa_H=\mu_H+1$ and
 $\kappa_{H''}=\nu_{H''}$ for $H'' \in \AAA \setminus \Set{H}$, and 
$\kappa' \in \LAT$ such that $\kappa_{H'}=\mu_{H'}+1$ and
 $\kappa_{H''}=\nu_{H''}$ for $H'' \in \AAA\setminus\Set{H'}$.
By the assumption, $\Delta(\kappa)=\Delta(\kappa')=0$.
Hence $\theta_\mu \not\in D(\AAA,\kappa)$ and $\theta_\mu \not\in D(\AAA,\kappa')$.
Let $\Set{\alpha_H \theta_\mu, \theta' }$ be a basis for 
the $S$-module $D(\AAA,\kappa)$.
Since $\theta_{\nu} \in D(\AAA,\nu) \subset D(\AAA,\kappa)$,
$\theta_{\nu} = a \alpha_H \theta_\mu + b \theta'$
for some $a,b \in \kk$.
If $\Set{\theta_{\mu}, \theta_{\nu} }$ is $S$-linearly dependent, then
$\theta_{\nu} = a \alpha_H \theta_\mu$.
Since $\theta_{\nu}(\alpha_{H'}) = a \alpha_H \theta_\mu(\alpha_{H'}) \in \ideal{ \alpha_{H'}^{\nu_{H'}}} = \ideal{ \alpha_{H'}^{\kappa'_{H'}}}$,
$\theta_\mu(\alpha_{H'})\in\ideal{ \alpha_{H'}^{\kappa'_{H'}}}$.
Hence $\theta_\mu \in D(\AAA,\kappa')$, which is contradiction.

Finally assume that 
$\mu_H +1=\nu_H$, $\mu_{H'}= \nu_{H'}+1$ 
for some $H,H' \in \AAA$ and 
$\mu_{H''}=\nu_{H''}$ for $H'' \in \AAA\setminus\Set{H, H'}$.
Let $\kappa = \mu \wedge \nu$ and $\kappa' = \mu \vee \nu$.
By the assumption, $\Delta(\kappa')=\Delta(\kappa)=0$.
Hence $\theta_\mu, \theta_\nu \not\in D(\AAA,\kappa')$. 
We may choose a basis $\Set{\theta_\mu, \theta'}$ for 
$D(\AAA,\kappa)$ such that 
$\Set{ \theta_\mu, \alpha_{H'} \theta'}$ 
is a basis for $D(\AAA,\mu)$.
Since $D(\AAA,\nu)\subset D(\AAA,\kappa)$, $\theta_\nu= a\theta_\mu +b \theta' $
for some $a,b \in \kk$.
If $\Set{\theta_{\mu}, \theta_{\nu} }$ is $S$-linearly dependent, then
$\theta_{\nu} = a \theta_\mu$.
Since $\theta_{\nu} = a \theta_\mu \in D(\AAA,\mu)\cap D(\AAA,\nu)$,
 $\theta_{\nu} \in D(\AAA,\kappa')$ which 
is a contradiction.
\end{proof}

\begin{lemma}\label{lemma:indep}
If $\mu,\nu \in \mpts{\Lat}$ satisfy 
$\distance{\mu}{\nu}=\Delta(\mu)+\Delta(\nu)$,
then $\Set{\theta_\mu, \theta_\nu}$ is $S$-linearly independent.
\end{lemma}
\begin{proof}
By the assumption, there exist some $\mu',\nu' \in \Lat$ such that 
\begin{itemize}
 \item $\distance{\mu'}{\nu'}=2$, 
 \item $\Delta(\kappa)=0$ for all $\kappa\in \LAT$ such that
 $\distance{\mu'}{\kappa}=\distance{\nu'}{\kappa}=1$,
 \item $\mu,\mu' \in C \in \fcc{\Lat}$, and 
 \item $\nu,\nu' \in C' \in \fcc{\Lat}$.
\end{itemize}
By Lemma \ref{lemma:independency} 
$\Set{\theta_{\mu'}, \theta_{\nu'}}$ is $S$-linearly independent.
Hence Lemma \ref{lemma:depincI} completes proof.
\end{proof}

\begin{proof}[Proof of Theorem \ref{theorem:independency}]
Apply Lemma \ref{lemma:depinc} and \ref{lemma:indep}.
\end{proof}

\subsection{Proof of Theorem \ref{theorem:basis}}

\begin{lemma}
\label{Lemma:lemma4.19}
Assume that $\mu, \nu\in {\Lat}$ satisfy 
$\Delta(\mu) + \Delta(\nu) = \distance{\mu}{\nu}$, and that 
$\Set{\theta_\mu, \theta_\nu}$ is $S$-linearly independent.
Then $\Set{\theta_\mu,\theta_\nu}$ is a basis for $D(\AAA,\mu\wedge\nu)$.
\end{lemma}
\begin{proof}
Since  $(\mu\wedge\nu)_H = \min\Set{\mu_H , \nu_H}$ for $H \in \AAA$,
\begin{gather*}
\numof{\mu\wedge\nu}=\sum_{H\in\AAA} \min\Set{\mu_H , \nu_H}.
\end{gather*}
On the other hand, 
\begin{align*}
 \deg(\theta_\mu) + \deg(\theta_\nu) 
&= \frac{\numof{\mu}-\Delta(\mu)}{2}+ \frac{\numof{\nu}-\Delta(\nu)}{2}\\
&= \frac{\numof{\mu}+\numof{\nu}-\Delta(\mu)-\Delta(\nu)}{2}\\
&= \frac{\numof{\mu}+\numof{\nu}-\distance{\mu}{\nu}}{2}\\
&= \sum_{H\in\AAA}\frac{\mu_H+\nu_H-|\mu_H-\nu_H|}{2}\\
&= \sum_{H\in\AAA}\min\Set{\mu_H , \nu_H} = \numof{\mu\wedge\nu}.
\end{align*}
Since $\mu\wedge\nu \subset \mu,\nu$, it follows 
from Saito's criterion
that 
$\Set{\theta_\mu,\theta_\nu}$ is a basis for $D(\AAA,\mu\wedge\nu)$.
\end{proof}

\begin{lemma}\label{lemma:basis}
Assume that $\mu, \nu \in {\Lat}$ satisfy 
$\Delta(\mu) + \Delta(\nu) = \distance{\mu}{\nu}$ and 
that $\Set{\theta_\mu, \theta_\nu}$ is $S$-linearly independent.
For $\kappa \in \LAT$ such that 
$\mu\wedge\nu \subset \kappa \subset \mu\vee\nu$, let us define 
\begin{align*}
\alpha_{\mu,\kappa} 
&= \prod_{H\in \AAA}   \alpha_H^{\max\Set{ \kappa_H - \mu_H , 0}}, \\
\alpha_{\nu,\kappa} 
&= \prod_{H\in \AAA}   \alpha_H^{\max\Set{ \kappa_H - \nu_H , 0}}. \\
\end{align*}
Then $\Set{\alpha_{\mu,\kappa} \theta_\mu,\alpha_{\nu,\kappa} \theta_\nu}$ is a basis for $D(\AAA,\kappa)$.
\end{lemma}
\begin{proof}
Note that 
$\deg(\alpha_{\mu,\kappa}) + \deg(\alpha_{\nu,\kappa})=\distance{\kappa}{\mu\wedge\nu}$ and that 
$\alpha_{\mu,\kappa} \theta_\mu$, $\alpha_{\nu,\kappa} \theta_\nu\in D(\AAA,\kappa)$. Thus Saito's 
criterion and  Lemma \ref{Lemma:lemma4.19} completes the proof.
\end{proof}

\begin{proof}[Proof of Theorem \ref{theorem:basis}]
By Theorem \ref{theorem:independency},
 $\Set{\theta_\mu,\theta_\nu}$ is  $S$-linearly independent.
Hence Lemma \ref{lemma:basis} completes the proof.
\end{proof}

\section{Application}\label{section:application}
In this section, we consider the case when a group acts on $V$. 
Let $W$ be a group acting on $V$ from the left.
Canonically, this action induces actions on $S$ and $\Der{\kk}{S}$,
i.e., $W$ acts on $S$ and $\Der{\kk}{S}$ by
$(\sigma f)(v) =f(\sigma^{-1}v)$ and
$(\sigma\delta)(f)=\sigma(\delta(\sigma^{-1}f))$ for 
$\sigma\in W$, $f\in S$, $\delta\in\Der{\kk}{S}$ and $v\in V$.
For each $\sigma\in W$, we assume $\AAA=\sigma \AAA$.
In this case, $W$ also acts on $\AAA$ as a subgroup of 
the symmetric group of $\AAA$.
Hence $W$ also acts on $\LAT$ by $(\sigma\mu)_H=\mu_{\sigma^{-1}H}$.
\begin{lemma}
\label{lemma:invariancity}
For $\mu\in \LAT$ and $\sigma\in W$,
$\Delta(\mu)=\Delta(\sigma\mu)$.
\end{lemma}
\begin{proof}
If $\Set{\theta,\theta'}$ be a homogeneous basis for $D(\AAA,\mu)$,
then  $\Set{\sigma\theta, \sigma\theta'}$ 
is a homogeneous basis for $D(\AAA,\sigma\mu)$.
\end{proof}

Next we assume that $\AAA^{W}=\emptyset$, i.e., 
for each $H\in\AAA$, there exists $\sigma_{H}\in W$ such that
$\sigma_{H}H\neq H$.

\begin{lemma}
\label{lemma:symmetricpeak}
Let $\mu\in \Lat$ satisfy $\sigma \mu = \mu$ for all $\sigma\in W$.
If there exist $\nu$ and $\kappa$ satisfying the following,
then $\mu\in\mpts{\Lat}$:
$\mu'\subset\nu$ for all $\mu\coveredby \mu'$;
$\Delta(\mu)-\Delta(\nu)> \distance{\mu}{\nu}-4$;
$\kappa\subset\mu'$ for all $\mu'\coveredby \mu$; and 
$\Delta(\mu)-\Delta(\kappa)> \distance{\kappa}{\mu}-4$.
\end{lemma}

\begin{proof}
It  suffices to show that $\Delta(\mu')<\Delta(\mu)$ 
if $\mu \coveredby \mu'$ or $\mu' \coveredby \mu$.
First let us assume $\mu \coveredby \mu'$,
$\Delta(\mu')>\Delta(\mu)$ and $\mu'_{H}\neq\mu_{H}$.
Since $\AAA^{W}=\emptyset$,
$H\neq\sigma H$ for some $\sigma\in W$.
For such $\sigma$, it holds that 
$(\sigma\mu')_{H}=\mu'_{\sigma^{-1} H}=\mu_{\sigma^{-1} H}=\mu_H \neq \mu'_{H}=\mu_{H}+1$, 
where the second equality holds because $\sigma^{-1}H \neq H$ and because of 
the definition of $\mu'$, the third because of the $W$-invariance of $\mu$. 
Hence $\sigma\mu' \neq \mu'$. 
By the same computation, we can show that 
\begin{eqnarray*}
(\sigma \mu')_{\sigma H}&=&\mu'_{\sigma H}+1\ \mbox{and}\\
(\sigma \mu')_{H'}&=&\mu'_{H'}\ (H' \in \AAA \setminus \{H,\sigma H\}).
\end{eqnarray*}
Hence $\distance{\sigma\mu'}{\mu'}=2$ and $\mu=\mu'\wedge \sigma\mu'$.
By the assumption
$\Delta(\mu')=\Delta(\sigma\mu')>\Delta(\mu)>0$.
Hence, by Lemma \ref{lemma:latA},
\begin{align*}
\Delta(\mu'\vee\sigma\mu') &= \Delta(\mu')+1\\
&= \Delta(\mu)+2.
\end{align*}
By Lemma \ref{lemma:s},
$\Delta(\mu'\vee\sigma\mu')= \Delta(\mu)+2 \leq \distance{\nu}{\mu'\vee\sigma\mu'}+\Delta(\nu)$.
Since $\distance{\nu}{\mu'\vee\sigma\mu'}+\Delta(\nu)=\distance{\nu}{\mu}-2+\Delta(\nu)$, we have 
$\Delta(\mu) -\Delta(\nu)\leq \distance{\nu}{\mu}-4$,
which is a contradiction.

The  same argument is valid for the case where
$\mu' \coveredby \mu$,
$\Delta(\mu')>\Delta(\mu)$ and $\mu'_{H}\neq\mu_{H}$.
Hence we have the lemma.
\end{proof}

As an application of the results above, we consider the exponents of 
Coxeter arrangements, which is a set of all reflecting hyperplanes 
of a finite irreducible Coxeter groups. Since $A_2$-type is 
investigated in \cite{W}, let us consider 
Coxeter arrangements of type $I_2(n)\ (n \ge 4)$. 


It is shown by Terao in \cite{T} that the constant multiplicity on the 
Coxeter arrangement 
is free and the exponents are also determined. We give the meaning of 
Terao's result from our point of view, i.e., the role of constant multiplicity 
in the multiplicity lattice.

\begin{proposition}\label{prop:b2g2}
Let $\AAA$ be a Coxeter arrangement of type $I_2(n)\ (n \ge 4)$. 
Then $\mu=(2k+1,\ldots,2k+1) \in \mpts\Lat$.
\end{proposition}
\begin{proof}
Let $\AAA$ be the Coxeter arrangement of type $I_2(n)$. Then 
we can take the Coxeter group of type $I_2(n)$ as $W$. 
Let $\nu=(2k+2,\ldots,2k+2)$ and $\kappa=(2k,\ldots,2k)$.
Then $\distance{\mu}{\nu}=\distance{\mu}{\kappa}=n$. 
Since $\Delta(\mu)=n-2$ and  $\Delta(\nu)=\Delta(\kappa)=0$ by \cite{T},
it follows from Lemma \ref{lemma:symmetricpeak} that $\mu\in \mpts{\LAT'}$.
%
\end{proof}

Now we can determine the basis and exponents of multiplicities on 
Coxeter arrangements when they are near the 
constant one, which is based on the primitive derivation methods in 
\cite{T} and \cite{Y0}. 

\begin{cor}
Let $\AAA$ be a Coxeter arrangement of type  $I_2(n)\ (n \ge 4)$,
$\mu=(2k+1,\ldots,2k+1) \in \LAT$
and $i\in \ZZ^{\numof{\AAA}}$ such that $|I|:=\sum_{H} |i_{H}| < \numof{\AAA}=n$.
If $\nu \in \Lambda$ is defined by $\nu_H=\mu_H+i_H$ and 
$I:=\sum_{H} i_H$, 
then 
\begin{align*}
\exp(\AAA,\nu)=
\left(
kn+1+\displaystyle \frac{I+|I|}{2},
(k+1)n-1+\displaystyle \frac{I-|I|}{2}
\right).
\end{align*}
\end{cor}

The proof of above corollary is completed by applying 
Corollary $\ref{cor:distance}$ and Proposition $\ref{prop:b2g2}$.

\begin{remark}
Recently in \cite{A3}, 
by using the results in this 
article, the first author proved that 
$\Delta(\mu) \le |\AAA|-2$ 
for $\mu \in P(\Lambda')$
in the case when 
a two-dimensional arrangement 
$\AAA$ is defined
over a field of characteristic zero. 
\end{remark}

\begin{remark}
In \cite{W10} it is proved that for the 
Coxeter multiarrangement $(\AAA,\mu)$ of type $B_2$ defined by 
$$
x^{2k+1}y^{2k+1}(x-y)^{2j+1}(x+y)^{2j+1}=0,
$$
it holds that $\Delta(\mu)=2$. However, to determine 
explicitly which multiplicity makes $\Delta=2$ is difficult even for 
$B_2$-type. 
\end{remark}



\end{document}